\documentclass[11pt]{article}
\usepackage{amsmath}
\usepackage{amsthm}
\usepackage{amsfonts}
\usepackage{amssymb}
\usepackage{latexsym}
\usepackage{diagbox}
\usepackage{mathtools}

\usepackage{array}

\usepackage{subcaption}
\usepackage{booktabs}
\usepackage{tkz-graph}
\usepackage{graphics,graphicx}
\usepackage{multirow}

\usepackage{tikz}
\usepackage{tikz-cd}

\usetikzlibrary{matrix,arrows}
\usetikzlibrary{fit}
\usetikzlibrary{patterns}
\usepackage{bm}
\usetikzlibrary{chains,fit,shapes}
\usetikzlibrary{positioning,calc}
\usetikzlibrary{graphs}
\usetikzlibrary{arrows,decorations.markings}
\usepackage[normalem]{ulem} 
\usepackage{hyperref}
\hypersetup{
  colorlinks=true,
  linkcolor=black,
  citecolor=black,
  urlcolor=black
}
\newcommand{\ulhref}[2]{\href{#1}{\uline{#2}}}

\newtheorem{thm}{Theorem}[section]  
\newtheorem{lem}[thm]{Lemma}      
\newtheorem{prop}[thm]{Proposition}
\newtheorem{cor}[thm]{Corollary}
\newtheorem{conj}[thm]{Conjecture}

\begin{document}
\begin{center}
{\large \bf Dowling's polynomial conjecture for independent sets of matroids}
\end{center}
\begin{center}
Shiqi Cao$^{1}$, Keyi Chen$^{2}$, Yitian Li$^{3}$ and Yuxin Wu$^{4}$\\[6pt]

$^{1,3}$Center for Combinatorics, LPMC, Nankai University, Tianjin 300071, P. R. China

$^{2}$ Changchun Boshuo School, Changchun 130103, P.R. China

$^{4}$School of Mathematical Sciences, LPMC, Nankai University, Tianjin 300071, P.R. China.
 Email:$^{1}${\tt shiqicao@mail.nankai.edu.cn},
       $^{2}${\tt chenkeyi2023@126.com} 
       $^{3}${\tt yitianli@mail.nankai.edu.cn},
       $^{4}${\tt yuxinwu@mail.nankai.edu.cn}
\end{center}

\noindent\textbf{Abstract.}
The celebrated Mason's conjecture states that the sequence of independent set numbers of any matroid is log-concave, and even ultra log-concave. The strong form of Mason's conjecture was independently solved by Anari, Liu, Oveis Gharan and Vinzant, and by Br\"and\'en and Huh. The weak form of Mason's conjecture was also generalized to a polynomial version by Dowling in 1980 by considering certain polynomial analogue of independent set numbers. In this paper we completely solve Dowling's polynomial conjecture by using the theory of Lorentzian polynomials.

	
\section{Introduction}\label{intro-sec}

The concept of matroids was introduced by Whitney \cite{whitney1935} in order to capture the fundamental properties of dependence that are common to graphs and matrices. In recent years, much attention has been drawn to studying various inequalities satisfied by combinatorial sequences associated to matroids. The main objective of this paper is to prove a polynomial conjecture proposed by Dowling in 1980 \cite{dowling1980}, which naturally implies Mason's log-concavtiy conjecture on the sequence of independent set numbers of a matroid. 

Let us first review some related background. Recall that a matroid $M$ is an ordered pair $(E,\mathcal{I})$ consisting of a finite set $E$ and a collection $\mathcal{I}$ of subsets of $E$, which satisfies the following three properties~\cite{Oxl11}:
\begin{enumerate}
  \item[(1)] $\emptyset \in \mathcal{I}$.
  \item[(2)] \emph{(hereditary\ property)} If $A \in \mathcal{I}$ and $A' \subseteq A$, then $A' \in \mathcal{I}$.
  \item[(3)] \emph{(exchange\ property)} If $A_1$ and $A_2$ are in $\mathcal{I}$ and $|A_1|<|A_2|$, then there exists an element
  $e \in A_2\setminus A_1$ such that $A_1 \cup \{e\} \in \mathcal{I}$.
\end{enumerate}
Each subset of $\mathcal{I}$ is called an \emph{independent set} of $M$. It is known that all maxiaml independent sets have the same size. Each maximal independent set is called a \emph{basis} of $M$. The \emph{rank} of $A\subseteq E$, denoted by $r_M(A)$, is defined to be the maximal size of independent subsets contained in $A$. 
The rank of $M$ is defined to be $r_M(E)$.
By abuse of notation, we simply write $r_M$ for $r_M(E)$.
In 1972 Mason \cite{mason1972} proposed the following conjecture on the independent sets of $M$.  
\begin{conj}\label{masonthm}
Given a matroid $M=(E,\mathcal{I})$ with $|E|=n$, let $I_k$ denote the number of independent sets of size $k$ for any $0\leq k\leq r_M$. Then for any $1\leq k\leq r_M-1$ we have
\begin{enumerate}
  \item[(i)] $I_k^2 \ge I_{k-1}\cdot I_{k+1}$ \emph{(log-concavity)},
  \item[(ii)] $I_k^2 \ge \left(1+\frac{1}{k}\right)\, I_{k-1}\cdot I_{k+1}$,
  \item[(iii)] $I_k^2 \ge \left(1+\frac{1}{k}\right)\left(1+\frac{1}{n-k}\right)\, I_{k-1}\cdot I_{k+1}$
  \emph{(ultra log-concavity)}.
\end{enumerate}
\end{conj}

Note that the above three inequalities are of increasing strength. The first complete proof of (i) in Conjecture \ref{masonthm} was given by Adiprasito, Huh and Katz \cite{AHK18}, who developed a combinatorial Hodge theory for matroids. For partial results on Mason's log-concavity conjecture, we refer the reader to the references cited in \cite{AHK18}. The second inequality in Conjecture \ref{masonthm} was proved by Huh, Schr\"oter and Wang \cite{HSW22} for any matroid. The third inequality in Conjecture \ref{masonthm}, known as Mason's ultra log-concavity conjecture, was independently proved by Anari, Liu, Oveis Gharan and Vinzant ~\cite{Ana18}, and by Br{\"a}nd{\'e}n and Huh~\cite{Bra20}. Both of these two groups used the same theory of certain class of polynomials, which was called completely log-concave polynomials in \cite{Ana18} while Lorentzian polynomials in \cite{Bra20}. Later, Chan and Pak~\cite{Cha22} gave another proof of Mason's ultra log-concavity conjecture by developing the theory of combinatorial atlas. For partial progress on the ultra log-concavity, we refer the read to the references cited in \cite{Ana18}, \cite{BH18} and  \cite{Bra20}.

We would like to point out that there is another conjecture  stronger than Mason's log-concavity conjecture, which was proposed by Dowling \cite{dowling1980} in 1980. 
To state Dowling's conjecture, we first introduce a partial order on $\mathbb{R}[x_1,x_2,\ldots,x_n]$, the ring of multivariate polynomials in $x_i$'s with real coefficients. Given two polynomials $f,g\in \mathbb{R}[x_1,x_2,\ldots,x_n]$, we say that $f\geq g$ if $f-g$ is a polynomial with nonnegative coefficients. Given a matroid $M=(E,\mathcal{I})$ with $|E|=n$, 
for any $0\le k\le r_M$ define
\begin{eqnarray}\label{fk}
f_k(M)=\sum_{I\in \mathcal{I},\,|I|=k}\left(\prod_{x_i\in I} x_i\right).
\end{eqnarray}
Dowling \cite{dowling1980} proposed the following conjecture.
\begin{conj}[Dowling's polynomial conjecture]\label{dowlingconj}
Let $M$ and $f_k(M)$ be defined as above. Then
\begin{eqnarray}\label{eq1}
f_k^2(M) \ge f_{k-1}(M)\,f_{k+1}(M)
\end{eqnarray}
holds for any $0< k< r_M$.
\end{conj}

It is clear that Conjecture \ref{dowlingconj} implies (i) of Conjecture \ref{masonthm}. Dowling proved his conjecture for $k\le 7$. Motivated by Conjecture \ref{dowlingconj} and (ii) of Conjecture \ref{masonthm}, it is natural to consider the following conjecture, which was implicit in Zhao \cite{Zha85}.

\begin{conj}\label{zhaoconj}
For any matroid $M$ and $0< k< r_M$, we have
\begin{eqnarray}
f_k^2(M) \ge \left(1+\frac{1}{k}\right)\,f_{k-1}(M)\,f_{k+1}(M)\label{f_2}.
\end{eqnarray}
\end{conj}

Zhao \cite{Zha85} proved the above conjecture for $k\le 5$. In view of 
(iii) of Conjecture \ref{masonthm} we very much hope that the following  polynomial analogue of Mason's ultra log-concavity conjecture holds:
\begin{eqnarray*}
    f_k^2(M) &\ge& \left(1+\frac{1}{k}\right)\left(1+\frac{1}{n-k}\right)\,f_{k-1}(M)\,f_{k+1}(M).
\end{eqnarray*}
Unfortunately, this fails in general. Consider the uniform matroid $M=U_{2,4}$, whose ground set is $E=\{1,2,3,4\}$ and whose independent sets are all subsets of $E$ containing at most two elements. We find that
\[f_0(U_{2,4}) = 1,\ f_1(U_{2,4})=x_1+x_2+x_3+x_4,\] 
\[f_2(U_{2,4}) = x_1x_2+x_1x_3+x_1x_4+x_2x_3+x_2x_4+x_3x_4.\] 
It is clear that 
\[f_1^2(U_{2,4}) \not \ge (1+\frac{1}{2})(1+\frac{1}{4-2})f_0(U_{2,4})f_2(U_{2,4}).\]

The main contribution of this paper is the proof of Dowling's polynomial conjecture. In fact, we directly prove Conjecture \ref{zhaoconj}, which implies Conjecture \ref{dowlingconj}. Our proofs of these two conjectures will be presented in Section \ref{s3}. 
We would like to point out Conjecture \ref{dowlingconj} and Conjecture \ref{zhaoconj} were independently proved by 
Ardila-Mantilla,  Cristancho, Denham, Eur, Huh and Wang \cite{Ard26}, where they originally attributed these two conjectures to Pak \cite{pak26}.
Section \ref{s2} is devoted to the introduction of some related concepts and results which will be used in subsequent sections. In Section \ref{s4} we prove some further inequalities satisfied by $f_k(M)$, which generalize Conjecture \ref{dowlingconj} and Conjecture \ref{zhaoconj}.

\section{Preliminaries}\label{s2}

In this section we recall Dowling's original approach to Conjecture 
\ref{dowlingconj}, as well as Zhao's equivalent characterization of Conjecture \ref{zhaoconj}. We also give an overview of the theory of Lorentzian polynomials, which plays an important role in the proof of our main results. 

Dowling's approach to Conjecture \ref{dowlingconj} involves the dual, deletion and contraction of matroids. Suppose that $M=(E,\mathcal{I})$ is a matroid with ground set $E$ and independent set family $\mathcal{I}$.
Following Oxely \cite{Oxl11}, the \emph{dual matroid} of $M$ is denoted by 
$M^*=(E,\mathcal{I}^*)$, whose ground set is still $E$ and whose bases are the complements of the bases of $M$. For any subset $T\subset E$, 
let $M \backslash T$ be the matroid obtained from $M$ by deleting $T$, whose ground set is $E\backslash T$ and whose independent sets are those subsets of $E\backslash T$ which are also independent in $M$.
The \emph{contraction} of $T$ from $M$, given by $M/T=(M^*\backslash T)^*$. 
The matroid $M\backslash (E\backslash T)$ is also called the \emph{restriction} of $M$ to $T$, denoted by $M(T)$.
More definitions and background on matroids can be found in~\cite{Oxl11}.

For any two disjoint subsets $X,Y\subseteq E$, let $M(X\cup Y)$ be the restriction of $M$ to $X\cup Y$, and let $M(X\cup Y)/Y$ be the minor obtained from $M(X\cup Y)$ by contracting $Y$. For the minor $M(X\cup Y)/Y$, its size is the cardinality of $X$, and its $depth$ in $M$ is the rank of $Y$.

Suppose that $N$ is a matroid on a set $Y$ with $|Y|=2k$. 
Given an ordered pair $(i,j)$ with $i+j=2k$, an independent $(i,j)$-partition of $N$ is an ordered partition $(A,B)$ of $Y$ such that both $A$ and $B$ are independent in $N$ with $|A|=i$ and $|B|=j$. Let $\pi_{i,j}(N)$ denote the number of independent $(i,j)$-partitions of $N$.

Dowling~\cite{dowling1980} proved the following result. 

\begin{prop}[{\cite[Proposition 1]{dowling1980}}]\label{dowlprop}
Given a finite matroid $M$ and a positive integer $l$, the inequality 
\begin{equation*}
f_l^{2}(M)\ge f_{l-1}(M)\,f_{l+1}(M)
\end{equation*}
holds, if and only if, for every $k\le l$ and every minor $N$ of $M$ of size $2k$ and depth $l-k$,
\begin{equation}\label{eq-Dowling}
\pi_{k,k}(N)\ge \pi_{k-1,k+1}(N).
\end{equation}
\end{prop}

In the same manner, Zhao~\cite{Zha85} gave an equivalent characterization of Conjecture \ref{zhaoconj} as follows. 

\begin{lem}[{\cite[Lemma 2]{Zha85}}]\label{zhaothm}
Given a finite matroid $M$ and a positive integer $l$, the inequality 
\begin{eqnarray*}
f_l^{2}(M)\ge \left(1+\frac{1}{l}\right) f_{l-1}(M)\,f_{l+1}(M)
\end{eqnarray*}
holds, if and only if, for every $k\le l$ and every minor $N$ of $M$ of size $2k$ and depth $l-k$,
\begin{eqnarray}\label{zhaoeq}
\pi_{k,k}(N)\ge \left(1+\frac{1}{l}\right)\pi_{k-1,k+1}(N).
\end{eqnarray}
\end{lem}

Next we recall the theory of Lorentzian polynomials, which was developed by Br{\"a}nd{\'e}n and Huh \cite{Bra20}. Let $n$ and $d$ be nonnegative integers and set $[n]=\{1,2,\ldots,n\}$. For any $i\in [n]$, let $\partial_{x_i}$, or simply $\partial_i$ if no confusion arises,  denote the partial derivative operator that maps
a polynomial $f\in\mathbb{R}[x_1,\ldots,x_n]$ to its partial derivative with respect to $x_i$. The Hessian of $f$, denoted by $H_f$, is defined as
\[H_f:= (\partial_i\partial_j\, f)_{i,j=1}^n.\]
For any $n$-tuple $\bm{\alpha}=(\alpha_1,\ldots,\alpha_n)\in \mathbb{N}^n$ of nonnegative integers, let
$x^{\bm{\alpha}}=x_1^{\alpha_1}\cdots x_n^{\alpha_n}$ and $\partial^{\bm{\alpha}}=\partial_1^{\alpha_1}\cdots \partial_n^{\alpha_n}$ as usual. 
If $f=\sum_{\bm\alpha} c_{\bm\alpha} x^{\bm\alpha}$, then its support is defined to be 
$$\operatorname{supp}(f):=\{\,\bm\alpha\in\mathbb{N}^n:\ c_{\bm\alpha}\neq 0\}.$$
In this paper we also use the usual notation $[x^{\bm\alpha}]f$ to represent the coefficient $c_{\bm\alpha}$. 
A subset $\mathcal{J} \subseteq \mathbb{N}^n$ is said to be $M$-convex if, for every $\bm\alpha,\bm\beta\in \mathcal{J}$ and any $i\in[n]$ such that $\alpha_i>\beta_i$, there exists $j\in [n]$ satisfying $\alpha_j<\beta_j$ such that $\bm{\alpha}-\mathbf{e_i}+\mathbf{e_j}\in \mathcal{J}$, where $\mathbf{e_i}$ and $\mathbf{e_j}$ are standard basis vectors. 
A homogeneous polynomial $f$ of degree $d$ with nonnegative coefficients is said to be \emph{Lorentzian} if $\operatorname{supp}(f)$ is $\mathrm{M}$-convex and for any $\bm\alpha\in \mathbb{N}^n$ satisfying $\sum_{i=1}^n\alpha_i=d-2$ the Hessian $\partial^{\bm\alpha}f$ has at most one positive eigenvalue.

It turns out that the Lorentzian property of a polynomial $f$ is closely related to the hyperbolicity of its Hessian $H_f$. Recall that an $n\times n$ matrix $A$ is called hyperbolic, if
\begin{align*}
    \langle \mathbf{v},\,A \mathbf{w}\rangle^2 \ge \langle \mathbf{v},\,A\mathbf{v}\rangle \langle \mathbf{w},\,A\mathbf{w}\rangle,\tag{Hyp}
\end{align*}
for every $\mathbf{v},\,\mathbf{w}\in \mathbb{R}^n$ with $\langle \mathbf{w},\,A\mathbf{w}\rangle >0$,
where $\langle \,,\, \rangle$ stands for the ordinary dot product of $\mathbb{R}^n$. 
Following Chan and Pak \cite{Cha22}, we state the following basic fact about Lorentzian polynomials, which was established by Br{\"a}nd{\'e}n and Huh in an equivalent form (see {\cite[Theorem 2.16(2)]{Bra20}}).  

\begin{thm}[{\cite[Theorem 5.2]{Cha22}}]\label{hyp}
    If $f \in \mathbb{R}[x_1,\dots,x_n]$ is a Lorentzian polynomial, then the Hessian $H_f$ satisfies (Hyp) for every $(x_1,\dots, x_n)\in \mathbb{R}_{>0}^n$.
\end{thm}

Br{\"a}nd{\'e}n and Huh \cite{Bra20} showed that, for homogeneous polynomials, the class of Lorentzian polynomials coincides with that of strongly log-concave polynomials introduced by Gurvits \cite{Gur09} and that of completely log-concave polynomials introduced by Anari, Oveis Gharan and Vinzant \cite{AnaI}. 
A polynomial $f\in\mathbb{R}[x_1,\ldots,x_n]$ with nonnegative coefficients is said to be \emph{strongly log-concave} if, for any $\bm\alpha\in \mathbb{N}^n$, the polynomial 
$\partial^{\bm\alpha}(f)$ is identically zero or $\log(\partial^{\bm\alpha}(f))$ is concave on $R_{>0}^n$. 
For a vector $\mathbf{v}\in\mathbb{R}^n$, let $\mathbf{D}_{\mathbf{v}}$ denote the directional derivative operator in direction $\mathbf{v}$, namely,
$\mathbf{D}_\mathbf{v}=\sum_{i=1}^n v_i\,\partial_i$. A polynomial $f\in\mathbb{R}[x_1,\ldots,x_n]$ is said to be \emph{completely log-concave} if for every set of nonnegative vectors
$\mathbf{v^{(1)}},\ldots,\mathbf{v^{(k)}}\in\mathbb{R}^n_{\ge 0}$, the polynomial $\mathbf{D}_{\mathbf{v^{(1)}}}\cdots \mathbf{D}_{\mathbf{v^{(k)}}}(f)$  is identically zero or it is nonnegative and log-concave over $\mathbb{R}^n_{\ge 0}$.

Br{\"a}nd{\'e}n and Huh~\cite{Bra20} established the following result.

\begin{thm}[{\cite[Theorem 2.30]{Bra20}}]\label{dengjia}
For any homogeneous polynomial $f\in\mathbb{R}[x_1,\ldots,x_n]$ with nonnegative coefficients the following conditions are equivalent: 
\begin{enumerate}
  \item[(1)] $f$ is completely log-concave;
  \item[(2)] $f$ is strongly log-concave;
  \item[(3)] $f$ is Lorentzian.
\end{enumerate}
\end{thm}

Let $\mathrm{L}_n^d$ be the set of homogeneous Lorentzian polynomials in $n$ variables of degree $d$. 
Br{\"a}nd{\'e}n and Huh~\cite{Bra20} also provided a large class of linear operators that preserve the Lorentzian property. 
For our purpose, we need the following three results.

\begin{thm}[{\cite[Theorem 2.10]{Bra20}}]\label{suanzi}
Suppose that $n,m$ are positive integers, $\mathbf{x}=(x_1,\ldots,x_n)$ and $\mathbf{y}=(y_1,\ldots,y_m)$. If $f(\mathbf{x})\in \mathrm{L}_n^d$, then $f(A\mathbf{y})\in L_m^d$ for any $n\times m$ matrix $A$ with nonnegative entries.
\end{thm}

\begin{cor}[{\cite[Corollary 2.11]{Bra20}}]\label{piandao}
If $f\in \mathrm{L}_n^d$, then $\mathbf{D}_\mathbf{v}(f) \in \mathrm{L}_n^{d-1}$ for any nonnegative vector $\mathbf{v}\in \mathbb{R}_{\geq 0}^n$. 
\end{cor}

\begin{cor}[{\cite[Corollary 2.32]{Bra20}}]\label{product-preserve-L}
If $f\in \mathrm{L}_n^d$ and $g\in \mathrm{L}_m^e$, then $fg\in \mathrm{L}_{m+n}^{d+e}$. 
\end{cor}

We also need a result due to Anari, Liu, Oveis Gharan, and Vinzant \cite{Ana18}, which plays an important role in their proof of Mason's ultra-log-concavity conjecture and can be restated as follows. 

\begin{thm}[{\cite[Theorem 4.1]{Ana18}}]\label{xiangcheng}
Suppose that $M$ is a matroid with ground set $[n]$ and independent set family $\mathcal{I}$. Then 
\begin{eqnarray}\label{eq-GM-x}
G_M(x,x_1,\ldots,x_n)=\sum_{I\in\mathcal{I}} x^{\,n-|I|}\prod_{i\in I} x_i
\end{eqnarray}
is a Lorentzian polynomial in $\mathbb{R}[x,x_1,\ldots,x_n]$.
\end{thm}


\section{Proof of Dowling's conjecture}\label{s3}

In this section we aim to give a proof of Dowling's polynomial conjecture. Since Conjecture \ref{zhaoconj} implies Conjecture \ref{dowlingconj}, we will directly prove the former conjecture. 

As shown in Proposition \ref{dowlprop} and Lemma \ref{zhaothm}, 
both Conjecture \ref{dowlingconj} and Conjecture \ref{zhaoconj} are equivalent to some inequalities satisfied by the number of independent set bipartitions of the ground sets of matroids. Let us first interpret these numbers as the coefficients of some polynomials associated with $G_M(x,x_1,\ldots,x_n)$ defined by \eqref{eq-GM-x}. For notational convenience, set
\[ \mathbf{x} = (x, x_1, \ldots, x_{n}),\qquad \ \mathbf{y} = (y, y_1, \ldots, y_{n}),\]
 and
\[ G_M(\mathbf{x}) = G_M(x,x_1,\ldots,x_n),\qquad \ G_M(\mathbf{y}) = G_M(y,y_1,\ldots,y_n).\]
We define a linear operator $\mathbf{S_i}$ on the polynomial ring $\mathbb{R}[\mathbf{x},\mathbf{y}]$, whose action on a polynomial $f\in \mathbb{R}[\mathbf{x},\mathbf{y}]$ is given by
\[ \mathbf{S}_i(f) = \left( \frac{\partial f}{\partial x_i} + \frac{\partial f}{\partial y_i} \right) \bigg|_{x_i = y_i = 0}. \]
Let $\mathbf{S} =\mathbf{S_1\cdots S_n}$. We have the following result.

\begin{lem}\label{lem-operator-s}
For any matroid $M=(E,\mathcal{I})$ of size $n$ and any $0\leq i\leq n$, there holds 
\begin{eqnarray}\label{eq-pi-interp}
\pi_{n-i,i}(M)=\pi_{i,n-i}(M)=[x^{n-i}y^{i}]\mathbf{S}(G_M(\mathbf{x})G_M(\mathbf{y})).
\end{eqnarray}
\end{lem}

\begin{proof}
By \eqref{eq-GM-x} a general term of $G_M(\mathbf{x})G_M(\mathbf{y})$ is 
$x^{n-|I|}y^{n-|I'|}\prod_{i\in I} x_i\prod_{j\in I'} y_j$, denoted by $g_{I,I'}$,
for some $I,I'\in\mathcal{I}$. If $I\cap I'\neq \emptyset$, say $k\in I\cap I'$, it is clear that $\mathbf{S_k}(g_{I,I'})=0$, and hence $\mathbf{S}(g_{I,I'})=0$. If $I\cup I'\neq [n]$, say $k\notin I\cup I'$, it is also clear that $\mathbf{S_k}(g_{I,I'})=0$, and hence $\mathbf{S}(g_{I,I'})=0$. When $I\cup I'=[n]$ and $I\cap I'= \emptyset$, i.e., $(I,I')$ is an independent bipartition of $M$, one can verify that $\mathbf{S}(g_{I,I'})=x^{n-|I|}y^{n-|I'|}$. This completes the proof. 
\end{proof}

The main result of this section is as follows.

\begin{thm}\label{strong}
If $M$ is a matroid of size $2k$, then 
\begin{eqnarray}\label{eq-main}
\pi_{k,k}(M) \ge \left(1+\frac{1}{k}\right)\pi_{k-1,k+1}(M).
\end{eqnarray}
\end{thm}

\begin{proof}
Theorem~\ref{xiangcheng} tells that 
$G_M(\mathbf{x})$ is a Lorentzian polynomial. By Corollary~\ref{product-preserve-L}, we see that $G_M(\mathbf{x})G_M(\mathbf{y})$ is also Lorentzian. By Theorem \ref{suanzi} and Corollary \ref{piandao}, each operator $\mathbf{S}_i$ preserves the Lorentzian property. (A special case of Theorem \ref{suanzi} implies that if $f(x_1,\ldots,x_{n-1},x_n)$ is Lorentzian so is $f(x_1,\ldots,x_{n-1},0)$.) Hence the polynomial $\mathbf{S}(G_M(\mathbf{x})G_M(\mathbf{y}))$ is a Lorentzian polynomial in $\mathbb{R}[x,y]$.

By Lemma \ref{lem-operator-s}, we get that 
\begin{align*}
\mathbf{S}(G_M(\mathbf{x})G_M(\mathbf{y}))&=\sum_{i=0}^{2k}\pi_{i,2k-i}x^iy^{2k-i}.
\end{align*}
It is routine to verify that 
\begin{align*}
    f_M(x,y):=&\frac{\partial^{2k-2}(\mathbf{S}(G_M(\mathbf{x})G_M(\mathbf{y})))}{\partial_x^{k-1}\partial_y^{k-1}}\\
=&\frac{1}{2}(k+1)!(k-1)!\pi_{k-1,k+1}x^2+k!k!\pi_{k,k}xy\\&+\frac{1}{2}(k+1)!(k-1)!\pi_{k-1,k+1}y^2,
\end{align*}
which is again Lorentzian, and its Hessian is
\begin{align*}
    H_{f_M}=\begin{bmatrix}
 (k-1)!(k+1)!\pi_{k-1,k+1} & k!k!\pi _{k,k}\\[5pt]
  k!k!\pi _{k,k}& (k-1)!(k+1)!\pi_{k+1,k-1}
\end{bmatrix}.
\end{align*}
By Theorem \ref{hyp}, $H_{f_M}$ satisfies (Hyp) for any $(x,y) \in \mathbb{R}_{\geq 0}^2$. Taking $\mathbf{v}=(1,0)$ and $\mathbf{w}=(0,1)$, we get 
\begin{align}
    &\langle \mathbf{v}, H_{f_M} \mathbf{w}\rangle^2 = (k!k!\pi_{k,k})^2,\label{vw21}
    \\[5pt]
    &\langle \mathbf{v}, H_{f_M} \mathbf{v}\rangle = (k-1)!(k+1)!\pi_{k-1,k+1},\label{vw22}\\[5pt]
    &\langle \mathbf{w}, H_{f_M} \mathbf{w}\rangle = (k-1)!(k+1)!\pi_{k+1,k-1}.\label{vw23}
\end{align}
Substituting (\ref{vw21}), (\ref{vw22}) and (\ref{vw23}) into (Hyp) leads to \eqref{eq-main}, as desired.
\end{proof}

Combining Theorem \ref{strong} and Lemma \ref{zhaothm}, we immediately obtain the following result, as conjectured in Conjecture \ref{zhaoconj} and independently proved in \cite[Theorem 1.6]{Ard26}.
\begin{cor}
For any matroid $M$ and $0< k< r_M$, we have
\begin{align*}
f_k^2(M) \ge \left(1+\frac{1}{k}\right)\,f_{k-1}(M)\,f_{k+1}(M).
\end{align*}
\end{cor}

As a corollary, we confirm Dowling's polynomial conjecture. 
\begin{cor}
For any matroid $M$ and $0< k< r_M$, we have
\begin{align*}
f_k^2(M) \ge f_{k-1}(M)\,f_{k+1}(M).
\end{align*}
\end{cor}

\section{Further generalization}\label{s4}

The aim of this section is to give a further generalization of Conjecture \ref{dowlingconj} and Conjecture \ref{zhaoconj}. The main result of this section is as follows. 

\begin{thm}\label{strongp}
Let $M$ be a matroid with rank $r_M$. Then, for any integers $p \ge 2$ and $p-1<l<r_M$, 
\begin{eqnarray}\label{highd}
    f_l^p(M)\ge \left(1+\frac{1}{(p-1)l}\right)\left(1+\frac{2}{(p-1)l}\right)\cdots\left(1+\frac{p-1}{(p-1)l}\right)f_{l+1}^{p-1}(M)f_{l-p+1}(M). 
\end{eqnarray} 
\end{thm}

Motivated by the proofs of Conjecture \ref{dowlingconj} and Conjecture \ref{zhaoconj}, we will give an equivalent characterization of \eqref{highd} in terms of independent set partitions of the ground set of the matroid $M=(E,\mathcal{I})$. 
For any positive integer $p\geq 2$ and $p$-tuple $\mathbf{i}=(i_1,i_2\ldots,i_p)\in\mathbb{N}_{>0}^p$ of positive integers, let $\pi_{\mathbf{i}}(M)$ denote the number of ordered set partitions $(A_1,\,A_2,\,\ldots,\,A_p)$ of $E$ with $A_k\in \mathcal{I}$ and $|A_k|=i_k$. Given any subset $X\subseteq E$ and any positive integer $k$, let $M \bigodot X^k$ denote the matroid obtained from $M$ by replacing each element of $X$ by $k$ elements in parallel. For each $x\in X$ we use $x^{(0)}=x, x^{(1)},\ldots,x^{(k-1)}$ to distinguish these $k$ elements. For each $0\leq i\leq k-1$ we set $X^{(i)}=\{x^{(i)}\,\mid\, x\in X\}$ and let $\varphi_{X^{(i)}}: X\rightarrow X^{(i)}$ to be the canonical map which sends each $x\in X$ to $x^{(i)}$. For any sequence $\mathbf{X}=(X_1,X_2,\ldots,X_p)$ of pairwise disjoint subsets of $E$ and any $\mathbf{q}=(q_1,q_2,\ldots,q_{p-1})\in\mathbb{P}^{p-1}$, let 
\begin{align*}
M[{\mathbf{X}},\mathbf{q}]&=\left(M(X_1\cup \cdots\cup X_p)\bigodot X_1^{q_1}\cdots\bigodot X_{p-1}^{q_{p-1}}\right)/X_p.
\end{align*}
The depth of the matroid $M[{\mathbf{X}},\mathbf{q}]$ is defined to be the rank of $X_p$ in $M$.
We have the following result. 

\begin{thm}\label{wyx_2}
    Suppose that $M$ is a matroid on the ground set $E$, $p\geq 2,l\geq 1$ are positive integers, and $(n_1,n_2,\dots,n_p)$, $(m_1,m_2,\dots,m_p)$
    are two $p$-tuples of integers satisfying $n_1+n_2+\cdots+n_p=m_1+m_2+\cdots+m_p=0$.
    Then
    \begin{align}\label{ntuiguang}
    f_{l+n_1}(M)f_{l+n_2}(M)\cdots f_{l+n_p}(M)\ge f_{l+m_1}(M)f_{l+m_2}(M)\cdots f_{l+m_p}(M)\end{align}
    if and only if for every $k\le l$ and every matroid of the form $N=M[{\mathbf{X}},\mathbf{q}]$ of size $pk$  and depth $l-k$, \begin{align}\label{qidong}
    \pi_{(k+n_1,k+n_2,\dots,k+n_p)}(N)\ge\pi_{(k+m_1,k+m_2,\dots,k+m_p)}(N),
    \end{align}
    where $\mathbf{X}=(X_1,X_2,\ldots,X_p)$ is a $p$-tuple of pairwise disjoint  subsets of $E$ and $\mathbf{q}=(1,2,\ldots,p-1)$.
\end{thm}

\begin{proof} Without loss of generality, we may assume that all of sub-indices in \eqref{ntuiguang} and \eqref{qidong} are nonnegative. 
Observe that each term of $f_{l+n_1}(M)f_{l+n_2}(M)\cdots f_{l+n_p}(M)$ is an integer multiple
of a monomial of the form
$$h=\prod_{x_{i_1}\in X_1}x_{i_1}\prod_{x_{i_2}\in X_2}x_{i_2}^2\cdots\prod_{x_{i_p}\in X_p}x_{i_p}^p,$$
for some pairwise disjoint subsets $X_1,X_2,\dots,X_p$ (some of them might be empty) satisfying $|X_1|+2|X_2|+\cdots+p|X_p|=pl$.
The coefficient of $h$ is equal to the number of $p$-tuples  
$(I_1,I_2,\dots,I_p)$ of independent subsets of $E$ such that $|I_i|=l+n_i$ and 
$$h=\prod_{x_{i_1}\in I_1}x_{i_1}\prod_{x_{i_2}\in I_2}x_{i_2}\cdots\prod_{x_{i_p}\in I_p}x_{i_p}.$$
In this case, for $1\leq j\leq p$ each element $x$ of $X_j$ is contained in exactly $j$ members of $\{I_1,I_2,\dots,I_p\}$, and hence it uniquely determines $j$ subsets $I_{i_1(x)},I_{i_2(x)},\dots,I_{i_j(x)}$ 
such that $x\in I_{i_1(x)}\cap I_{i_2(x)}\cap \cdots\cap I_{i_j(x)}$.
Let $|X_p|=l-k$. We proceed to show that there exists a bijection $\Phi$ between the set of such tuples 
$(I_1,I_2,\dots,I_p)$ and the set of  independent $(k+n_1,k+n_2,\dots,k+n_p)$-partitions $(I'_1,I'_2,\dots,I'_p)$ of the ground set of the matroid $M[{\mathbf{X}},\mathbf{q}]$. Given $(I_1,I_2,\dots,I_p)$, for each $1\leq s\leq p$ let $I'_s$ be obtained from $I_s\setminus X_p$ in the following way: if $x$ also lies in $X_j$ for some $j$ (this is also unique) and hence determines a $j$-tuple $(I_{i_1(x)},I_{i_2(x)},\dots,I_{i_j(x)})$ with $i_1(x)<i_2(x)<\cdots<i_j(x)$, then replace $x$ with $\varphi_{X^{(k-1)}}(x)$ provided that $i_k(x)=s$. A little thought shows that  each $I'_s$ is an independent set of cardinality $k+n_s$ in $M[{\mathbf{X}},\mathbf{q}]$ and if $s\neq t$ then $I'_s\cap I'_t=\emptyset$. 
Recall that the matroid  $M[{\mathbf{X}},\mathbf{q}]$ is of size
$|X_1|+2|X_2|+\cdots+(p-1)|X_{p-1}|=pk$.
The fact that $|I'_s|=k+n_s$ and $n_1+\cdots+n_p=0$ tells that $(I'_1,I'_2,\dots,I'_p)$ is an independent 
 $(k+n_1,k+n_2,\dots,k+n_p)$-partitions of 
$M[{\mathbf{X}},\mathbf{q}]$.  Let $\Phi((I_1,I_2,\dots,I_p))=(I'_1,I'_2,\dots,I'_p)$. We note that $\Phi$ has a bijective inverse map $\Phi^{-1}$ which can be constructed by replacing each element of the form $x^{(k)}$ in $I'_s\cup X_p$ by $x$.

Therefore, the coefficient of $h$ in $f_{l+n_1}(M)f_{l+n_2}(M)\cdots f_{l+n_p}(M)$ is $\pi_{k+n_1,k+n_2,\dots,k+n_p}(N)$. 
For the same reason, the coefficient of $h$ in $f_{l+m_1}f_{l+m_2}\cdots f_{l+m_p}$ is $\pi_{k+m_1,k+m_2,\dots,k+m_p}(N)$. This means that (\ref{qidong}) implies (\ref{ntuiguang}).

It remains to show that (\ref{ntuiguang}) implies (\ref{qidong}). Suppose (\ref{ntuiguang}) holds and let $N=M[\mathbf X,\mathbf q]$ be a matroid of size $pk$ and depth $l-k$. We may assume that $X_p$ is independent in $M$, so that $|X_p|=l-k$. Define a
monomial $h$ as above, and observe that the coefficients of $h$ in $f_{l+n_1}(M)f_{l+n_2}(M)\cdots f_{l+n_p}(M)$ and $f_{l+m_1}(M)f_{l+m_2}(M)\cdots f_{l+m_p}(M)$ are given by the left and right sides of (\ref{qidong}), respectively. Thus (\ref{qidong}) follows from (\ref{ntuiguang}).
\end{proof}

Now to prove Theorem \ref{strongp} it suffices to prove the following result. 

\begin{thm}\label{gaojie}
    Given positive integers $p,k>0$, let $\mathbf{k}$ and $\mathbf{\tilde{k}}$ be defined by\begin{align}\label{eq-k}
\mathbf{k}=(\underbrace{k,k,\ldots,k}_{p’s}), 
\qquad \mathbf{\tilde{k}}=(\underbrace{k+1,k+1,\ldots,k+1}_{(p-1)’s},k-(p-1)).
\end{align}
    If $M$ is a matroid of size $pk$, then 
    \begin{align*}
        \pi_{\mathbf{k}}(M)\ge \left(1+\frac{1}{(p-1)k}\right)\left(1+\frac{2}{(p-1)k}\right)\cdots\left(1+\frac{p-1}{(p-1)k}\right) \pi_{\mathbf{\tilde{k}}}(M).
    \end{align*}
\end{thm}

Given $p$ sets of mutually disjoint variables 

\[ \mathbf{x}_1=(x_1, x_{1,1}, \dots, x_{1,pk}), \dots,\mathbf{x}_p=(x_p,x_{p,1},\dots x_{p,pk}),\]

let
\[
 G_M(\mathbf{x}_j) = \sum_{I \in \mathcal{I}(M)}x_j^{pk-|I|} \prod_{i \in I} x_{j,i},  \]
and
\[G(\mathbf{x}_1,\dots,\mathbf{x}_p)=\prod_{j=1}^pG_M(\mathbf{x}_j).\]
We define a linear operator $\mathbf{H}_i$ on the polynomial ring $\mathbb{R}[x_1,\dots,x_p ]$, whose action on a polynomial $f\in \mathbb{R}[x_1,\dots,x_p ]$ is given by
\[\mathbf{H}_i(f) = \sum_{j=1}^p\frac{\partial f}{\partial x_{j,i}}\bigg|_{x_{j,i} = 0}.\]

Let $\mathbf{H} =\mathbf{H_1\cdots H_{pk}} $. We have the following result, whose proof is similar to that of Lemma \ref{lem-operator-s} and omitted here. 

\begin{lem}\label{duiyingp}
    For any matroid $M = (E, I)$ of size $n$, any positive integer $p\geq 2$ and any $\mathbf{i}=(i_1,i_2,\dots,i_p)\in \mathbb{N}_{>0}^p$ with $i_1+i_2+\cdots +i_p=n$, we have
    \begin{align*}
        \pi_{\mathbf{i}}(M) = [x_1^{n-i_1}x_2^{n-i_2}\cdots x_p^{n-i_p}]\mathbf{H}(G(\mathbf{x}_1,\dots,\mathbf{x}_p)),
    \end{align*}
    and moreover, if  $\mathbf{j}=(j_1,j_2,\dots,j_p)$ is any permutation of $\mathbf{i}$ then
    $\pi_{\mathbf{i}}(M) = \pi_{\mathbf{j}}(M).$
\end{lem}

We proceed to prove Theorem \ref{gaojie}.

\begin{proof}[Proof of Theorem \ref{gaojie}]
Theorem~\ref{xiangcheng} tells that 
$G_M(\mathbf{x}_j)$ is a Lorentzian polynomial. By Corollary~\ref{product-preserve-L}, we see that $G(\mathbf{x}_1,\dots, \mathbf{x}_p)$ is also Lorentzian. By Theorem \ref{suanzi} and Corolloary \ref{piandao}, each operator $\mathbf{H}_i$ preserves the Lorentzian property. Hence the polynomial $\mathbf{H}(G(\mathbf{x}_1,\dots, \mathbf{x}_p))$ is a Lorentzian polynomial in $\mathbb{R}[x_1,x_2,\dots,x_p]$.

By Lemma \ref{duiyingp}, we get that
\[P_M(x_1,x_2,\dots,x_p) = \mathbf{H}(G(\mathbf{x}_1,\dots, \mathbf{x}_p)) = \sum_{{\mathbf{i}=(i_1,i_2,\dots,i_p)\in\mathbb{N}_{>0}^p}\atop{i_1+i_2+\cdots+i_p=pk}}\pi_{\mathbf{i}}x_1^{pk-i_1}x_2^{pk-i_2}\cdots x_p^{pk-i_p}.\]

For any $1\leq t\leq p-2$, it is routine to verify that
\begin{align*}
P^{(t)}=&\partial_1^{(p-1)k-1}\cdots \partial_t^{(p-1)k-1}\partial_{t+1}^{(p-1)k}\cdots\partial_{p-2}^{(p-1)k}\bigg|_{x_{1}=\cdots=x_{p-2}=0}P_M(x_1,\dots,x_p)\\
=&((pk-k-1)!)^t((pk-k)!)^{p-2-t}\sum_{i=0}^{2k}\pi_{\mathbf{k}(t,i)}x_{p-1}^{pk-i}x_p^{(p-2)k+t+i},
\end{align*}
where $\mathbf{k}(t,i)=(\underbrace{k+1,\dots,k+1}_{t's},\underbrace{k,\dots,k}_{(p-2-t)'s},i,2k-t-i)$.
It is clear that $P^{(t)}$ is Lorentzian. 
Since the sequence of coefficients of a bivariate Lorentzian polynomial is ultra log-concave, the sequence
$$\left\{\dfrac{\pi_{\mathbf{k}(t,i)}}{\binom{2(p-1)k+t}{(p-2)k+t+i}}\right\}_{i=0}^{2k}$$
is unimodal. Note that the mode of the sequence is $k-\left \lfloor t/2\right \rfloor $. Letting $i_1=k$ and $i_2=k+1$, we get that
$$\dfrac{\pi_{\mathbf{k}(t,i_1)}}{\binom{2(p-1)k+t}{(p-1)k+t}}\ge\dfrac{\pi_{\mathbf{k}(t,i_2)}}{\binom{2(p-1)k+t}{(p-1)k+t+1}}.$$
In view of Lemma \ref{duiyingp}, we obtain
$$\pi_{k+1,\dots,k+1,k,\dots,k,k-t}\ge\left(1+\dfrac{t+1}{(p-1)k}\right)\pi_{k+1,\dots,k+1,k+1,k,\dots,k,k-t-1}.$$
Iteration of the above inequality leads to 
$$\pi_{k,\dots,k}\ge\left(1+\dfrac1{(p-1)k}\right)\left(1+\dfrac2{(p-1)k}\right)\cdots\left(1+\dfrac{p-1}{(p-1)k}\right)\pi_{k+1,\dots,k+1,k-p+1},$$
as desired. 
\end{proof}

Combining Theorem \ref{wyx_2} and Theorem \ref{gaojie}, we immediately obtain Theorem \ref{strongp}.


\section*{Acknowledgements.}
We wish to thank Karim Adiprasito, Bishal Deb and Shouda Wang for helpful discussions.  
This work is supported by the Fundamental Research Funds for the Central Universities.

\end{document}